\documentclass[12pt,reqno]{amsart}
\usepackage{amsmath,amsthm,amsfonts,amssymb,latexsym,mathrsfs,color}
\usepackage{mathdots,multirow}
\usepackage{arydshln}
\usepackage{enumerate, latexsym, amsmath, amsfonts, amssymb, amsthm, color}
\def\pmod #1{\ ({\rm{mod}}\ #1)}
\def\Z{\Bbb Z}
\def\N{\Bbb N}

\def\l{\left}
\def\r{\right}

\def\t{\text}
\def\f{\frac}

\def\bi{\binom}

\def\eq{\equiv}

\def\da{\delta}

\def\Proof{\noindent{\it Proof}}
\def\Ack{\medskip\noindent {\bf Acknowledgments}}
\theoremstyle{plain}
\newtheorem{theorem}{Theorem}

\newtheorem{lemma}{Lemma}

\theoremstyle{definition}

\theoremstyle{remark}

\def \N{\mathbb{N}}

 \vspace{4mm}

\begin{document}
 \baselineskip=16pt
\hbox{Int. J. Number Theory 14(2018), no.\,5, 1265--1277.}
\medskip

\title
[Hankel-type determinants for some combinatorial sequences]
{Hankel-type determinants for some combinatorial sequences}

\author[Bao-Xuan Zhu and Zhi-Wei Sun] {Bao-Xuan Zhu and Zhi-Wei Sun}

\thanks{Both authors are supported by the National Natural Science Foundation of China (grants 11571150 and 11571162 respectively)}

\address {(Bao-Xuan Zhu) School of Mathematical Sciences, Jiangsu Normal
University, Xuzhou  221116, People's Republic of China}
\email{bxzhu@jsnu.edu.cn}

\address {(Zhi-Wei Sun) Department of Mathematics, Nanjing
University, Nanjing 210093, People's Republic of China}
\email{{\tt zwsun@nju.edu.cn}}


 \begin{abstract} In this paper we confirm several conjectures of Sun on Hankel-type determinants for some combinatorial sequences
 including Franel numbers, Domb numbers and Ap\'ery numbers.
 For any nonnegative integer $n$, define
 \begin{gather*}f_n:=\sum_{k=0}^n\bi nk^3,\ D_n:=\sum_{k=0}^n\bi nk^2\bi{2k}k\bi{2(n-k)}{n-k},
 \\b_n:=\sum_{k=0}^n\bi nk^2\bi{n+k}k,\ A_n:=\sum_{k=0}^n\bi nk^2\bi{n+k}k^2.
 \end{gather*}
For $n=0,1,2,\ldots$, we show that $6^{-n}|f_{i+j}|_{0\leq i,j\leq n}$ and $12^{-n}|D_{i+j}|_{0\le i,j\le n}$
are positive odd integers, and $10^{-n}|b_{i+j}|_{0\leq i,j\leq n}$ and $24^{-n}|A_{i+j}|_{0\leq i,j\leq n}$ are always integers.
\end{abstract}

\maketitle

\section{Introduction}
\setcounter{lemma}{0} \setcounter{theorem}{0}
\setcounter{corollary}{0} \setcounter{remark}{0}
\setcounter{equation}{0} \setcounter{conjecture}{0}

For a sequence $(a_n)_{n\ge 0}$ of complex numbers,  its {\it Hankel
matrix} is given by
$$H=[a_{i+j}]_{i,j\ge 0}=
\left[
  \begin{array}{ccccc}
    a_0 & a_1 & a_2 & a_3 & \cdots \\
    a_1 & a_2 & a_3 & a_4 & \\
    a_2 & a_3 & a_4 & a_5 & \\
    a_3 & a_4 & a_5 & a_6 & \\
    \vdots &  &  &  & \ddots \\
  \end{array}
\right].$$   Hankel
matrices are related to orthogonal polynomials, moment sequences and
continued fractions, and they have been extensively studied in many
branches of mathematics (see, e.g., \cite{KS,ST43}).
The Hankel-type determinants for the sequence $a_0,a_1,a_2,\ldots$
are those determinants $|a_{i+j} |_{0\le i,j\le n}$ with $n\in\N=\{0,1,2,\ldots\}$, which are sometimes called
Tur\'{a}nian determinants (cf. Karlin and Szeg\"{o}~\cite{KS}). For evaluations of Hankel-type determinants,
LU decomposition, continued fractions and Dodgson
condensation are some of the available tools that have been used
with considerable success. See Krattenthaler \cite{Kra99,Kra05}  for a wide range of
techniques used to evaluate certain Hankel-type determinants.

In this paper we study positivity and divisibility properties of certain Hankel-type determinants for some well-known combinatorial sequences,
and confirm several conjectures of Sun \cite{S13c}.

Recall that the Franel numbers are defined by
$$f_n:=\sum_{k=0}^n\binom{n}{k}^3\ \ (n=0,1,2,\ldots).$$
For $r=3,4,5,\ldots$, the $r$-th order Franel numbers are given by
$$f_n^{(r)}:=\sum_{k=0}^n\bi nk^r\ \ (n=0,1,2,\ldots).$$
In 2013 Sun \cite{S13a,S13b} proved some fundamental congruences involving Franel numbers, for example, he showed that
for any prime $p>3$ we have
$$\sum_{k=0}^{p-1}(-1)^kf_k\eq\l(\f p3\r)\pmod p,\ \sum_{k=0}^{p-1}\f{(-1)^k}kf_k\eq0\pmod{p^2},$$
and
\begin{align*}\sum_{k=0}^{p-1}\f{f_k}{2^k}
\eq\begin{cases}2x-p/(2x)\pmod{p^2}&\mbox{if}\  p=x^2+3y^2\ (x,y\in\Z)\ \mbox{with}\ 3\mid x-1,
\\3p/\binom{(p+1)/2}{(p+1)/6}\pmod{p^2}&\mbox{if}\ p\eq2\pmod 3.\end{cases}.
\end{align*}

Our first theorem is about Hankel-type determinants for Franel numbers or generalized Franel numbers.

\begin{theorem}\label{Th1.1} Let $n\in\N$ and $r\in\{3,4,\ldots\}$. Then $2^{-n}|f^{(r)}_{i+j}|_{0\leq i,j\leq n}$ is an odd
integer. Furthermore, $6^{-n}|f_{i+j}|_{0\leq i,j\leq n}$ is a positive odd integer.
\end{theorem}

The Domb numbers defined by
$$D_n=\sum_{k=0}^n\binom{n}k^2\binom{2k}k\bi{2(n-k)}{n-k}\ \ (n\in\N)$$
have various combinatorial interpretations, for example, $D_n$ is the number of $2n$-step polygons on the diamond lattice.
The so-called Catalan-Larcombe-French numbers are given by
$$P_n=\sum_{k=0}^n\frac{\binom{2k}{k}^2\binom{2(n-k)}{n-k}^2}{\binom{n}{k}}\ (n\in\N).$$
Both Domb numbers and Catalan-Larcombe-French numbers are related to Ramanujan-type series for $1/\pi$
(cf. \cite{CCL,CC}).

Our second result is about Hankel-type determinants for Domb numbers as well as Catalan-Larcombe-French numbers.

\begin{theorem}\label{Th1.2} For $n\in\N$, both
$12^{-n}|D_{i+j}|_{0\leq i,j\leq n}$ and
$2^{-n(n+3)}|P_{i+j}|_{0\leq i,j\leq n}$ are positive odd integers.
\end{theorem}

The integers
$$b_n=\sum_{k=0}^n\binom{n}{k}^2\binom{n+k}{k}\ \ (n=0,1,2,\ldots)$$ and
$$A_n=\sum_{k=0}^n\binom{n}{k}^2\binom{n+k}{k}^2\ \ (n=0,1,2,\ldots)$$
are two kinds of Ap\'ery numbers. They were first introduced by R. Ap\'ery \cite{A}
in his proofs of the irrationality of $\zeta(2)=\sum_{n=1}^\infty1/n^2=\pi^2/6$ and $\zeta(3)=\sum_{n=1}^\infty1/n^3$.
Such numbers are also related to Ramanujan-type series for $1/\pi$
(cf. \cite{CC}). For congruences involving Ap\'ery numbers $A_n\ (n\in\N)$, see \cite{S11}.

Now we state our last theorem.

\begin{theorem}\label{Th1.3}
For $n\in\N$, both $10^{-n}|b_{i+j}|_{0\leq i,j\leq n}$ and $24^{-n}|A_{i+j}|_{0\leq i,j\leq n}$ are always integers.
\end{theorem}

Theorems 1.1-1.3 were originally conjectured by
the second author \cite{S13c}. We are not able to prove Sun's conjecture (\cite{S13c})
that both $|b_{i+j}|_{0\leq i,j\leq n}$ and $|A_{i+j}|_{0\leq i,j\leq n}$ are always positive.

We will show Theorems 1.1-1.3 in Sections 2-4 respectively.

\section{Proof of Theorem \ref{Th1.1}}
\setcounter{lemma}{0} \setcounter{theorem}{0}
\setcounter{corollary}{0} \setcounter{remark}{0}
\setcounter{equation}{0} \setcounter{conjecture}{0}


The binomial transformation of a sequence $(x_k)_{k\ge0}$ of numbers is the sequence
$(x_n')_{n\ge0}$ with
\begin{equation}\label{2.1}x_n':=\sum_{k=0}^n\bi nk x_k.
\end{equation}
This often rises in combinatorics.

The following basic result is well-known and we will use it frequently.

\begin{lemma}\cite[Theorem 1]{La01}\label{le+lay}
Let $(x_k)_{k\ge0}$ be a sequence of numbers. For any $n\in\N$, we have
\begin{equation}\label{2.2}|x_{i+j}|_{0\leq i,j\leq n}=|x'_{i+j}|_{0\leq i,j\leq n}.
\end{equation}
\end{lemma}

Let $A=[a_{n,k}]_{n,k\ge 0}$ be a matrix of real
numbers. It is called {\it totally positive} ({\it TP} for short) if
all its minors are nonnegative. Total positivity of matrices plays
an important role in various branches of mathematics such as statistics,
probability, mechanics, economics, and computer science (see, e.g.,
\cite{Pin10,Kar68}).  A sequence $(a_k)_{k\ge
0}$ of numbers is called a {\it Stieltjes moment} sequence if its
Hankel matrix $H$ is TP. It is well known that $(a_k)_{k\ge0}$ is a
Stieltjes moment sequence if and only if we can write a general term in the form
\begin{equation}\label{i-e}
a_k=\int_0^{+\infty}x^kd\mu(x),
\end{equation}
where $\mu$ is a non-negative measure on $[0,+\infty)$ (see, e.g.,
\cite[Theorem 4.4]{Pin10}). To determine whether a sequence has the Stieltjes moment property or not is
one of classical moment problems and it arises naturally in many
branches of mathematics (cf. \cite{ST43,Wi41}). There are many transformations and convolutions of sequences preserving
Stieltjes moment sequences, see, e.g.,  \cite{WZ}. We need the following lemma in this direction.

\begin{lemma}\cite{WZ}\label{lem+transformation}
If both $(x_n)_{n\geq 0}$ and $(y_n)_{n\geq 0}$ are Stieltjes moment
sequences, then so is the sequence $(w_n)_{n\ge0}$, where
$$w_n:=\sum_{k=0}^{n}\binom{n}{k}^2x_ky_{n-k}.$$
\end{lemma}
The next result plays an important role in our proof.
\begin{lemma}\label{lem+odd} Let $k$ be a positive integer. Suppose that
$(x_i)_{i\geq0}$ is an integer sequence for which $x_0=1$, $2k\mid x_i$ for all $i\geq1$, and
$4k\mid x_i$ if and only if $i$ is not a power of two.  Then, for any $n\in\N$,
the number $(2k)^{-n}|x_{i+j}|_{0\leq i,j\leq n}$ is
an odd integer.
\end{lemma}
\begin{proof} Since $2k\mid x_m$ for all $m=1,2,3,\ldots$, by the Laplace expansion of $|x_{i+j}|_{0\leq i,j\leq n}$ according to the first row, it suffices to
show that
\begin{eqnarray*}
|x_{i+j}/(2k)|_{1\leq i,j\leq n}\equiv1 \pmod 2.
\end{eqnarray*}
For any positive integer $m$, clearly $x_m/(2k)$ is congruent to $1$ or $0$ modulo 2 according as $m$ is a power of two or not.
Let $B_n$ denote the $(0,1)$-matrix $[x_{i+j}/(2k) \pmod 2]_{1\leq i,j\leq n}$. It suffices to show the claim that
\begin{equation}\label{pm}|B_{n}|\in\{\pm1\}.\end{equation}

(\ref{pm}) is trivial for $n=1,2,3$. Below we let $n\ge4$.
If $n+1$ is a power of two, then
the matrix $B_{n}$ is an upper triangular matrix with the anti-diagonal line containing no $0$, and thus
$|B_{n}|=(-1)^{\binom{n}{2}}=-1$.
If $n$ is a power of two, then
$$B_{n}=\left[
   \begin{array}{ccccccc}
      &  &  &  &  &  1 & 0\\
      &  &  &  &  1  & 0 & 0\\
      &  &  & \iddots & \vdots &\vdots & \vdots \\
      0 & 1  &  & \cdots & 0 & 0 & 0\\
      1 & 0  &  & \cdots & 0 & 0 & 0\\
      0 & 0  &  & \cdots & 0 & 0 & 1 \\
   \end{array}
 \right]$$
and $|B_{n}|\in\{\pm1\}$.

Now we suppose that $n=2^m+t$ for some $t=1,\ldots,2^m-2$. Note that $2t+2\le 2^m+t=n$ and $n-2t\le 2^m-1<n-t$.
For any $i,j\in\{1,\ldots,n\}$ with $i+j\ge2^m$, clearly $i+j$ is a power of two if and only if
the ordered pair $(i,j)$ is among
$$(2^m-r,r)\ (r=1,\ldots,2^m-1)\ \mbox{and}\ (n-s,n-2t+s)\ (s=0,\ldots,2t).$$
Therefore $B_{n}$ has the following form with the upper-left-most submatrix of order $n-2t-1$ identical with $B_{n-2t-1}$.
\[
\begin{array}{c@{\hspace{-5pt}}l}
\left[\begin{array}{cccccc;{2pt/2pt}ccccccc}
& & & & & & & & 1 & & & \\
& & & & & & & \iddots & & & \\
& & & & & & 1 & & & & & \\
& & & & & 1 & & & & \multicolumn{2}{c}{\multirow{2}*{{\Huge0}}} \\
& & & & \iddots & & & & & & \\
& & & 1 & & & & & & & &\\
\hdashline[2pt/2pt]
& & 1 & & & & & & & & & & 1\\
& \iddots & & & & & & \multicolumn{2}{c}{\multirow{2}*{{\Huge0}}} & &  & \iddots\\
1 & & & & &  & & & &  &  1 & &\\
\hdashline[2pt/2pt]
& & & & & &  & &  & 1 & & &\\
\hdashline[2pt/2pt]
& & \multicolumn{2}{c}{\multirow{2}*{{\Huge0}}} & & &  & & 1 & & \multicolumn{2}{c}{\multirow{2}*{{\Huge0}}} & \\
& & & & & & & \iddots & & & &\\
& & & & & & 1 & & & & &\\
\end{array}\right]
&\begin{array}{l}
\left.\rule{0mm}{20mm}\right\}{n-2t-1}\\
\left.\rule{0mm}{10mm}\right\}{t}\\
\left.\rule{0mm}{2mm}\right\}{\mbox{row}\ 2^m}\\
\left.\rule{0mm}{10mm}\right\}{t}\\
\end{array}\\[-5pt]
\begin{array}{cc}
\underbrace{\rule{38mm}{0mm}}_{n-2t-1}&
\underbrace{\rule{43mm}{0mm}}_{2t+1}\end{array}&
\end{array}.
\]

For $1\le i\le t$ and $n-2t\le j\le 2^m-1$ with $i+j\le 2^m$, clearly $1\le 2^m-j\le 2^m-n+2t=t$, if the $(i,j)$-entry is $1$ then
we let row $i$ subtract row $2^{m+1}-j=2^m+(2^m-j)$ to make the $(i,j)$-entry become $0$.
Similarly, for $n-2t\le i<2^m=n-t$ and $1\le j\le t$ with $i+j\le 2^m$, we let column $j$ subtract column $2^{m+1}-i=(n-t)+2^m-i$ to make the $(i,j)$-entry become $0$.
Thus the determinant of $B_{n}$ is equal to that
of the next matrix:

\[
\begin{array}{c@{\hspace{-5pt}}l}
\left[\begin{array}{cccccc;{2pt/2pt}ccccccc}
& & & & & & & &  & & & \\
& & & & & & & & & & \\
& & & & & &  & & & & & \\
& & & & & 1 & & & & \multicolumn{2}{c}{\multirow{2}*{{\Huge0}}} \\
& & & & \iddots & & & & & & \\
& & & 1 & & & & & & & &\\
\hdashline[2pt/2pt]
& &  & & & & & & & & & & 1\\
& & & & & & & \multicolumn{2}{c}{\multirow{2}*{{\Huge0}}} & &  & \iddots\\
& & & & &  & & & &  &  1 & &\\
& & \multicolumn{2}{c}{\multirow{2}*{{\Huge0}}} & &  & &  & & 1 & & &\\
& & & & & &  & & 1 & & \multicolumn{2}{c}{\multirow{2}*{{\Huge0}}} & \\
& & & & & & & \iddots & & & &\\
& & & & & & 1 & & & & &\\
\end{array}\right]
&\begin{array}{l}
\left.\rule{0mm}{18mm}\right\}{n-2t-1}\\
\left.\rule{0mm}{22mm}\right\}{2t+1}\\
\end{array}\\[-5pt]
\begin{array}{cc}
\underbrace{\rule{30mm}{0mm}}_{n-2t-1}&
\underbrace{\rule{45mm}{0mm}}_{2t+1}\end{array}&
\end{array}.
\]
It follows that $$|B_{n}|=|B_{n-2t-1}|\times
\begin{vmatrix}
 &  &  &  &  &  1 \\
 &  \multicolumn{2}{c}{\multirow{2}*{{\Huge0}}} &  &  1  & \\
 &  &  & \iddots &  & \\
 &  1  &  &  \multicolumn{2}{c}{\multirow{2}*{{\Huge0}}}  &  \\
 1 &  &  &  &  &  \\
\end{vmatrix}_{(2t+1)\times(2t+1)}.$$
So
$$|B_{n-2t-1}|\in\{\pm1\}\Rightarrow |B_n|\in\{\pm1\}.$$
Therefore (\ref{pm}) holds by induction, and this concludes our proof of Lemma \ref{lem+odd}.
\end{proof}

We also need a result of N. J. Calkin \cite{C98}.
\begin{lemma}
\cite[Lemma 12]{C98}\label{lem+mul} For any positive integers $r$ and $n$,
we have $2^{\ell(n)}\mid f_n^{(r)}$, where $\ell(n)$ denotes the numbers of $1$'s in
the binary expansion of $n$.
\end{lemma}

\medskip
\noindent
{\it Proof of Theorem \ref{Th1.1}}. For any positive integer $m$, by Lemma \ref{lem+mul} we have
$2\mid f_m^{(r)}$, and $4\mid f_m^{(r)}$ if $m$ is not a power of two.
When $m$ is a power of two, we have
$$\bi mk=\f mk\bi{m-1}{k-1}\eq0\pmod2\quad\mbox{for all}\ k=1,\ldots,m-1,$$
hence
$$f_m^{(r)}\eq\bi m0^r+\bi mm^r=2\pmod{2^r}$$
and thus $4\nmid f_m^{(r)}$. Applying Lemma \ref{lem+odd} to the sequence $(f_m^{(r)})_{m\ge0}$, we see that
the number $2^{-n}|f^{(r)}_{i+j}|_{0\leq i,j\leq n}$ is always an odd integer.

 By Barrucand's identity (cf. \cite{B}), for any $m\in\N$,
 $$f_m':=\sum_{k=0}^m\binom{m}{k}f_k$$
coincides with
$$g_m:=\sum_{k=0}^m\binom{m}{k}^2\binom{2k}{k}.$$
(See also \cite[A002893]{Sl}, and \cite{S16} for an extension.)
Moreover, by Lemma \ref{lem+transformation}, $(g_m)_{m\ge0}$ is a Stieltjes moment sequence since
$(\binom{2m}{m})_{m\geq0}$ is a Stieltjes moment sequence (cf.
\cite{CK}).
Combining this with Lemma~\ref{le+lay}, we have
$$|f_{i+j}|_{0\leq i,j\leq n}=|g_{i+j}|_{0\leq i,j\leq n}\ge0.$$
By Fermat's little theorem, $a^3\eq a\pmod 3$ for any $a\in\Z$. Thus, for any positive integer $m$ we have
$$g_m=\sum_{k=0}^m\bi mk f_k\eq\sum_{k=0}^m\bi mk\sum_{j=0}^k\bi kj=\sum_{k=0}^m\bi mk2^k=3^m\eq0\pmod 3.$$
So $|g_{i+j}|_{0\leq i,j\leq n}$ is divisibly by $3^n$.

In view of the above, $6^{-n}|f_{i+j}|_{0\leq i,j\leq n}$ is always a positive odd
integer, as desired. \qed

\section{Proof of Theorem \ref{Th1.2}}
\setcounter{lemma}{0} \setcounter{theorem}{0}
\setcounter{corollary}{0} \setcounter{remark}{0}
\setcounter{equation}{0} \setcounter{conjecture}{0}


\begin{lemma}\label{lem+center binom} Let $m$ and $n$ be positive integers, and set
$$D^{(m)}_n:=\sum_{k=0}^n\binom{n}{k}^m\binom{2k}{k}\binom{2(n-k)}{n-k}.$$
Then $4\mid D^{(m)}_n$. Also, $8\mid D^{(m)}_n$ if and only if $n$ is not a power of two.
\end{lemma}
\Proof. For each $k=1,2,3,\ldots$ we obviously have
$$\binom{2k}{k}=2\binom{2k-1}{k-1}\eq0\pmod2.$$
Thus
\begin{align*}&\sum^n_{k=0\atop 2k\not=n}\bi nk^m\bi{2k}k\bi{2(n-k)}{n-k}
\\=&\sum_{k=0}^{\lfloor(n-1)/2\rfloor}\l(\bi nk^m+\bi{n}{n-k}^m\r)\bi{2k}k\bi{2(n-k)}{n-k}
\\=&2\sum_{k=0}^{\lfloor(n-1)/2\rfloor}\bi nk^m\bi{2k}k\bi{2(n-k)}{n-k}
\\\eq&2\bi n0^m\bi{2\times0}0\bi{2n}n=4\bi{2n-1}{n-1}\pmod 8.
\end{align*}
If $n=2k$ for some positive integer $k$, then
$$\bi nk^m\bi{2k}k\bi{2(n-k)}{n-k}=\bi{2k}k^{m+2}=2^{m+2}\bi{2k-1}{k-1}^{m+2}\eq0\pmod 8.$$
So we always have
$$D^{(m)}_n\eq4\bi{2n-1}{n-1}\pmod 8.$$

Now we show that $\bi{2n-1}{n-1}$ is odd if and only if $n$ is a power of two.
Clearly, $\bi{2-1}{1-1}=1$ is odd. For $n>1$, we can write $n-1$ as $\sum_{i=0}^k\da_i2^i$ with $\da_0,\ldots,\da_k\in\{0,1\}$ and $\da_k=1$,
hence
\begin{align*}\bi{2n-1}{n-1}=&\bi{\da_k\times2^{k+1}+\da_{k-1}\times2^k+\ldots+\da_02+1}{\da_k2^k+\ldots+\da_12+\da_0}
\\\eq&\bi{\da_k}0\bi{\da_{k-1}}{\da_k}\cdots\bi{\da_0}{\da_1}\bi1{\da_0}\pmod2
\end{align*}
by Lucas' congruence (cf. \cite{HS}), and thus
$$\bi{2n-1}{n-1}\eq1\pmod2\iff \da_0=\da_1=\ldots=\da_k=1\iff n \ \mbox{is a power of two}.$$

Combining the above, we immediately obtain the desired result. \qed

\begin{lemma}\cite[P\'olya
and Szeg\H o]{PS64}\label{lem+PSzeg} If both $(x_n)_{n\geq 0}$ and
$(y_n)_{n\geq 0}$ are Stieltjes moment sequences, then so is the
sequence $(z_n)_{n\ge0}$, where 
$$z_n=\sum_{k=0}^{n}\binom{n}{k}x_ky_{n-k}.$$
\end{lemma}

\begin{lemma}\label{lem+Dom} We have $D_n\equiv1\pmod 3$ for all $n=0,1,2,\ldots$.
\end{lemma}
\Proof. We use induction on $n$.

Clearly, both $D_0=1$ and $D_1=4$ are congruent to $1$ modulo $3$.

Now let $n>1$ be an integer and assume that $D_m\equiv1\pmod3$ for all $m=0,\ldots,n-1$.

{\it Case}\ 1. $3\mid n$.

In this case, by Lucas' theorem, for any $k=0,\ldots,n$ we have
$$\binom{n}{k}\equiv\begin{cases}\binom{n/3}{k/3}\pmod3&\mbox{if}\ 3\mid k,
\\0\pmod3&\mbox{if}\ 3\nmid k.\end{cases}$$
Thus
\begin{align*}D_n\equiv&\sum_{j=0}^{n/3}\binom{n/3}{j}^2\binom{6j}{3j}\binom{2(n-3j)}{n-3j}
\\\equiv&\sum_{j=0}^{n/3}\binom{n/3}j^2\binom{2j}j\binom{2(n/3-j)}{n/3-j}=D_{n/3}\equiv1\pmod 3
\end{align*}
with the help of the induction hypothesis.

{\it Case} 2. $3\nmid n$.

It is known that
$$n^3D_n=2(2n-1)(5n^2-5n+2)D_{n-1}-64(n-1)^3D_{n-2}$$
(cf. \cite[A002895]{Sl}) which can be obtained via Zeilberger's algorithm. So we have
\begin{align*}n^3D_n\equiv&(n+1)(-n^2+n+2)D_{n-1}-(n-1)^3D_{n-2}
\\\equiv&-(n+1)^3D_{n-1}-(n-1)^3D_{n-2}\pmod 3.
\end{align*}
By Fermat's little theorem, $a^3\eq a\pmod3$ for all $a\in\Z$. Thus, by applying the induction hypothesis we obtain
$$nD_n\eq-(n+1)-(n-1)\equiv n\pmod3$$
and hence $D_n\equiv1\pmod3$ as desired.

In view of the above, we have completed the proof Lemma
\ref{lem+Dom}.  \qed
\begin{lemma}\label{lem+Dom+3} For any positive integer $n$, we have $D_n''\eq0\pmod3$.
\end{lemma}
\Proof. By Lemma \ref{lem+Dom}, we have
$$D_k'=\sum_{j=0}^k\binom{k}jD_j\equiv\sum_{j=0}^k\binom{k}j\equiv2^k\pmod
3$$ for all $k\in\N$. Thus
$$D_n''=\sum_{k=0}^n\bi nk D_k'\equiv\sum_{j=0}^n\binom{n}k2^k\equiv3^n\equiv0\pmod3.$$ This completes the proof. \qed


\medskip
\noindent{\it Proof of Theorem \ref{Th1.2}}. (i) By Lemma \ref{lem+center binom}, the
sequence $(D_m)_{m\ge 0}$ satisfies the conditions of Lemma
\ref{lem+odd} with $k=2$. Thus, $4^{-n}|D_{i+j}|_{0\leq i,j\leq n}$ is always
an odd integer. On the other hand, it follows from Lemmas \ref{le+lay} and
\ref{lem+Dom+3} that $3^{-n}|D_{i+j}|_{0\leq i,j\leq n}$ is also an integer.
Hence $12^{-n}|D_{i+j}|_{0\leq i,j\leq n}$ is odd. Moreover, by
Lemma \ref{lem+transformation}, we know that $(D_m)_{m\geq0}$ is a
Stieltjes moment sequence since $(\binom{2m}{m})_{m\geq0}$ is a
Stieltjes moment sequence. So $12^{-n}|D_{i+j}|_{0\leq i,j\leq n}$
is always a positive odd integer.

(ii) By Lemma \ref{lem+PSzeg}, we know that $(D_m^{(1)})_{m\geq0}$
is a Stieltjes moment sequence since $(\binom{2m}{m})_{m\geq0}$ is a
Stieltjes moment sequence. So $|D^{(1)}_{i+j}|_{0\leq i,j\leq
n}$ is always nonnegative. Hence, with the help of Lemma
\ref{lem+center binom} and Lemma \ref{lem+odd},
$4^{-n}|D^{(1)}_{i+j}|_{0\leq i,j\leq n}$ is a positive
odd integer. It is known (cf. \cite[A053175]{Sl}) that
$$P_m=2^m\sum_{k=0}^{\lfloor m/2\rfloor}\binom{n}{2k}\binom{2k}{k}^24^{m-2k}=2^mD^{(1)}_m$$
for any $m\in\N$. Therefore,
\begin{align*}\f{|P_{i+j}|_{0\leq i,j\leq n}}{2^{n(n+3)}}=&\f{|2^{i+j}D^{(1)}_{i+j}|_{0\le i,j\le n}}{2^{n(n+3)}}
=\f{\prod_{i=1}^n2^i\times \prod_{j=1}^n 2^j}{2^{n(n+3)}}|D^{(1)}_{i+j}|_{0\le i,j\le n}
\\=&\f{2^{n(n+1)}}{2^{n(n+3)}}|D^{(1)}_{i+j}|_{0\le i,j\le n}=\f{|D^{(1)}_{i+j}|_{0\le i,j\le n}}{4^n},
\end{align*} which is a positive odd integer. \qed

\section{Proof of Theorem \ref{Th1.3}}
\setcounter{lemma}{0} \setcounter{theorem}{0}
\setcounter{corollary}{0} \setcounter{remark}{0}
\setcounter{equation}{0} \setcounter{conjecture}{0}

 \begin{lemma}\label{lem+apery b}
For any positive integer $n$, we have
$$b_n'\equiv0\pmod {2}\ \ \mbox{and}\ \
b_n''\equiv0\pmod {5}.$$
 \end{lemma}
 \Proof. For any positive integer $m$, we have
 $$b_m=\sum_{k=0}^m\bi mk\bi{m+k}{2k}\bi{2k}k=1+\sum_{k=1}^m\bi mk\bi{m+k}{2k}2\bi{2k-1}{k-1}\eq1\pmod 2.$$
 Thus
 $$b_n'=\sum_{k=0}^n\bi nkb_k\eq\sum_{k=0}^n\bi nk=2^n\eq0\pmod2.$$
If $b_{k}\equiv 3^k \pmod {5}$ for all $k\in\N$, then
 \begin{align*}
b_n''=&\sum_{m=0}^n\bi nm\sum_{k=0}^m\bi mkb_k
\\\equiv&\sum_{m=0}^n\binom{n}{m}\sum_{k=0}^m\bi mk3^k=\sum_{m=0}^n\bi nm4^m=5^n\eq0\pmod5.
\end{align*}

It remains to prove $b_k\eq 3^k\pmod5$ by induction on $k\in\N$.
Note that $b_k=3^k$ for $k=0,1$. Let $k>1$ be an integer.
It is well known that (cf. \cite[A005258]{Sl})
$$k^2b_k=(11k^2-11k+3)b_{k-1}+(k-1)^2b_{k-2}.$$
If $b_{k-1}\eq 3^{k-1}\pmod5$ and $b_{k-2}\eq 3^{k-2}\pmod5$, then
$$k^2b_k\eq (3(k^2-k+3)+(k-1)^2)3^{k-2}\eq9k^23^{k-2}=k^23^k\pmod5$$
and hence $b_k\eq 3^k\pmod5$ if $5\nmid k$.
When $5\mid k$, by Lucas' theorem, for each $j\in\{0,\ldots,k\}$ we have
$$\bi{k}j^2\bi{k+j}j\eq \begin{cases}\bi{k/5}{j/5}^2\bi{(k+j)/5}{j/5}\pmod5&\t{if}\ 5\mid j,
\\0\pmod5&\t{if}\ 5\nmid j.\end{cases}$$
So, if $5\mid k$ and $b_{k/5}\eq 3^{k/5}\pmod5$, then
$$b_k\eq\sum_{i=0}^{k/5}\bi{k/5}{i}^2\bi{k/5+i}i^2=b_{k/5}\eq 3^{k/5}\eq (3^5)^{k/5}=3^k\pmod5.$$
This concludes our proof. \qed

 \begin{lemma}\label{lem+apery +a}
 For each $n=3,4,\ldots$, we have $A_n'\equiv 0\pmod {24}$.
 \end{lemma}
  \Proof.
 From Gessel \cite{Ges}, we have  $A_{2k}\equiv1\pmod {8}$, $A_{2k+1}\equiv5\pmod {8}$ and $A_{k}\equiv(-1)^k\pmod {3}$ for
 all $k\in\N$. This implies that
 $A_k\equiv3-2(-1)^k\pmod {24}$ for all $k\in\N$.
Therefore,
$$
A_n'=\sum_{k=0}^n\binom{n}{k}A_k\equiv \sum_{k=0}^n\binom{n}{k}(3-2(-1)^k)= 3\times
2^{n}\equiv 0\pmod {24}.$$
This concludes the proof.
 \qed

\medskip
\noindent {\it Proof of Theorem \ref{Th1.3}}. (i) By Lemma \ref{le+lay} and Lemma
\ref{lem+apery b}, $10^{-n}|b_{i+j}|_{0\leq i,j\leq n}$ is
always an integer.

(ii) Let $x_n=A_n'=\sum_{k=0}^n\binom{n}{k}A_k$. Then $x_0=1$, $x_1=6$,
$x_2=84$, $x_3=1680=70\times 24.$ By Lemma \ref{le+lay}, we have
$$24^{-n}|A_{i+j}|_{0\leq i,j\leq n}=24^{-n}|x_{i+j}|_{0\leq i,j\leq
n}.$$ If a term $X$ in the
Laplace expansion of $|x_{i+j}|_{0\leq i,j\leq n}$ only contains one
entry in $\{x_0,x_1,x_2\}$, then by Lemma \ref{lem+apery +a} we have $X\equiv
0\pmod {24^n}$.
 So, we only need to consider all terms containing two
or three entries in $\{x_0,x_1,x_2\}$. Clearly,
$x_0x_2=84$, $x_1x_1=36$, $x_1x_2\equiv 0\pmod {24}$, $x_2x_2\equiv
0\pmod {24}$ and $x_2x_2x_2\equiv 0\pmod {24^2}$. Thus, among all
the terms in the Laplace expansion of $|x_{i+j}|_{0\leq i,j\leq n}$,
it suffices to consider those containing $x_0x_2$ and $x_1x_1$.
Hence, for some integer $M\equiv0\pmod{24^{n-1}}$ we have
$$|x_{i+j}|_{0\leq i,j\leq n}\equiv
(x_0x_2-x_1x_1)M\equiv0\pmod {24^n}.$$
This concludes the proof. \qed

\Ack. The authors would like to thank the referee for helpful comments.


\bigskip


\begin{thebibliography}{S13a}

\bibitem{A} R. Ap\'ery, Irrationalit\'e de $\zeta(2)$ et $\zeta(3)$,
Ast\'erisque 61 (1979) 11--13.
\bibitem{B} P. Barrucand, A combinatorial identity, problem 75-4,  SIAM Review 17 (1975) 168.
\bibitem{C98}
N.J. Calkin, Factors of sums of powers of binomial coefficients,
Acta Arith. 81 (1998), 17--26.
\bibitem{CCL} H. H. Chan, S. H. Chan and Z. Liu, Domb's numbers and Ramanujan-Sato type series
for $1/\pi$, Adv. in Math. 186 (2004) 396--410.
\bibitem{CC}H. H. Chan and S. Cooper, Rational analogues of Ramanujan¡¯s series for $1/\pi$, Math.
Proc. Cambridge Philos. Soc. 153 (2012) 361--383.
\bibitem{CK}
J. Cigler, C. Krattenthaler, Some determinants of path generating
functions, Adv. in Appl. Math.  46 (2011) 144--174.
\bibitem{Ges}
I. Gessel, Some  Congruences  for  Ap\'{e}ry  Numbers, J.
Number Theory 14 (1982) 362--368.
\bibitem{HS} H. Hu and Z.-W. Sun, An extension of Lucas' theorem, Proc. Amer. Math. Soc. 129 (2001) 3471--3478.
\bibitem{Kar68}
S. Karlin, Total Positivity, Vol.~I, Stanford Univ. Press,
Stanford, 1968.
\bibitem{KS}
S. Karlin, Szeg\H {o}, On certain determinants whose elements are
orthogonal polynomials, J. Analyse Math. 8 (1960/1961) 1--157.
\bibitem{Kra99}
C. Krattenthaler, Advanced determinant calculus, in: The Andrews
Festschrift, Sem. Lothar. Combin. 42 (1999), Article B42q, 67 pp.
\bibitem{Kra05}
C. Krattenthaler, Advanced determinant calculus: A complement,
Linear Algebra Appl. 411 (2005) 68--166.
\bibitem{La01}
J.W. Layman, The Hankel transform and some of its properties, J.
Integer Seq. 4 (2001), Article 01.1.5, 11pp.
\bibitem{Pin10}
A. Pinkus, Totally Positive Matrices, Cambridge University Press,
Cambridge, 2010.
\bibitem{PS64}
G. P\'{o}lya, G. Szeg\"o, Problems and Theorems in Analysis, Vol.
II, 3rd ed., Springer, Berlin, 1998.

\bibitem{ST43}
J.A. Shohat, J.D. Tamarkin, The Problem of Moments. Amer. Math.
Soc., New York, 1943.
\bibitem{Sl} N.J.A. Sloane,  {\rm Sequences A002895 and A005258 in OEIS (On-Line Encyclopedia of Integer Sequences)}, {\tt http://oeis.org}.

\bibitem{S11} Z.-W. Sun, On sums of Apery polynomials and related congruences,
¡¡¡¡J. Number Theory 132 (2012) 2673--2699.
\bibitem{S13a} Z.-W. Sun, Congruences for Franel numbers, Adv. in Appl. Math. 51 (2013) 524--535.
\bibitem{S13b} Z.-W. Sun, Connections between $p=x^2+3y^2$ and Franel numbers, J. Number Theory 133 (2013) 2914--2928.
\bibitem{S13c} Z.-W. Sun, On some determinants with Legendre symbol entries, arXiv:1308.2900, (2013).
\bibitem{S16} Z.-W. Sun, Congruences involving $g_n(x)=\sum_{k=0}^n\bi nk^2\bi{2k}kx^k$, Ramanujan J. 40 (2016) 511--533.

\bibitem{WZ}
Y. Wang, B.-X. Zhu, Log-convex and Stieltjes moment sequences, Adv.
in Appl. Math. 81 (2016) 115--127.
\bibitem{Wi41}
D.V. Widder, The Laplace Transform, Princeton Univ. Press, Princeton, 1941.
\end{thebibliography}
\end{document}